\documentclass[12pt]{article}
\usepackage{epsfig}
\usepackage{amsmath,amssymb,amsthm}
\usepackage{graphicx}
\usepackage{color}
\usepackage{cite}   
                                                                                 
\newtheorem{theo}{Theorem}
\newtheorem{definition}{Definition}
\newtheorem{proposition}{Proposition}
\newtheorem{lemme}{Lemma}
\newtheorem{corollary}{Corollary}

\newtheorem{rems}{Remarks}

\title{First passage time law for some L\'evy processes with compound Poisson: existence 
of a conditional density with incomplete observation\footnote{I thank M. Pontier for her careful reading.}}
\author{Waly NGOM\footnote{ IMT University of Toulouse, France, waly.ngom@math.univ-toulouse.fr
-F.S.T University of Dakar, S\'en\'egal, waly.ngom@ucad.edu.sn}}
\begin{document}
\maketitle
\begin{abstract}
\noindent
We study the default risk in incomplete information. That means, we model the value of a
firm by one L\'evy process which is the sum of Brownian motion with drift and compound
Poisson process. This L\'evy process can not be observed completely and we let 
another process which represents the available information on the firm. We obtain an equation 
satisfied by the conditional density of the default time given the available information. 
\end{abstract}
\textbf{Acknowledgemnts.} I thank my Ph.D advisors Laure Coutin and Papa Ngom for their 
help and pointing out error.This work is supported by A.N.R. Masterie.

\section{Introduction}
In our work, we study a first passage time of a level $x>0$ by a jump diffusion process $X$ which respectively 
models default and assets of a firm. We investigate the behavior of the default time under incomplete observation
of assets. Such a study is very 
important when failure of a large industrial company or some great political decision taken by the parliement can affect
the dividend policy of the issuing firms. In the litterature, we find some papers in relation to this topic.
Duffie and Lando in \cite{duffielando} suppose that bond investors cannot observe the issuer's assets directly and 
receive instead only periodic and imperfect reports. For a setting in which the assets of the firm are a geometric Brownian
motion until informed equityholders optimally liquidate, they derive the conditional distribution of the assets, given
accounting data and survivorship.
Dorobantu  \cite{Dorobantu} provides the intensity function of the default time. That is very important for investors, but the 
information brought by this intensity is low.
Volpi et al. \cite{volpi} prove that the Laplace transform of the random triple (first passage time, overshoot, undershoot)
satisfies some kind of integral equation and after normalization of the first passage time, they show under some 
assumptions that the triple random converges in distribution as $x$ goes to $\infty.$
In \cite{gapeev}, the authors study a model of a financial market in which the dividend rates of two risky assets change 
their initial values to other constant ones at the times at which certain unobservable external events occur. The asset 
price dynamics are described by geometric Brownian motion with random drift rates switching at exponential random times which
are independent of each other and of the constantly correlated driving Brownian motion. They obtain closed expressions
for rational values of European contingent claims through the filtering estimates of the occurence of switching times 
and their conditional probability density derived given the filtration generated by the underlying asset price process.
Coutin and Dorobantu in \cite{Coutin-Dorobantu}  prove that the law of the default time has a density (defective when 
$\mathbb{E}(X_{1})< 0$) with respect to the Lebesgue measure 

for PAIS with compound Poisson process and not null Brownian motion

We extend this approach by using some filtering theory.
The purpose of our paper is to add to these studies the behavior of the conditional law of a first passage time
by a L\'evy process with compound Poisson process given partial information.
The paper is organized as follow: In Section 2, we recall the model and the filtering framework. In Section 3, we show the 
existence of the density and in Section 4, we show that
the conditional density satisfies some kind of integro-differential equation.To finish, we give some technical proofs
in the Appendix.

\section{Models and filtering framework}
In this section, we recall the model and the results of Coutin and Dorobantu \cite{Coutin-Dorobantu} and we present the filtering 
framework of Pardoux \cite{Pardoux} or Coutin \cite{Coutin}.
\subsection{Model}
Let $(\Omega,  \mathcal{F}, (\mathcal{F}_{t})_{t\geq 0}, \mathbb{P})$ be a filtered probability space .\\
Let $\tilde{X}$ be a Brownian motion with drift $ m\in \mathbb{R}$ and for $z> 0, \tilde{\tau}_{z}=
\inf\{ t\geq 0, \tilde{X}_{t} \geq z\}.$ By $(5.12)$ page 197 of \cite{Karatzas}, $\tilde{\tau}_{z}$ has the following
law on $\bar{\mathbb{R}}_{+}:$
\begin{equation}
 \tilde{f}(u,z)du+\mathbb{P}(\tilde{\tau}_{z}=\infty)\delta_{\infty}(du)
\end{equation}
where
\begin{equation*}
 \tilde{f}(u,z)=\frac{\mid z \mid}{\sqrt{2\pi u^{3}}}\exp[-\frac{1}{2u}(z-mu)^{2}]{\bf 1}_{]0,+\infty[}(u)
 \mbox{ and }\mathbb{P}(\tilde{\tau}_{z}=\infty)=1-e^{mz-\mid mz \mid}.
\end{equation*}
The function $\tilde{f}(.,z)$ is  $C^{\infty} $ on $]0,+\infty[$, 
and all its derivatives admit 0 as right limit at 0 and then belongs to $\mathcal{C}^{\infty}$ .
Let $X$ be a  L\'evy process defined as following
\begin{equation}
  X_{t} = mt + W_{t} + \sum_{i=1}^{N_{t}}Y_{i} \quad t\in \mathbb{R}^{+}.
\end{equation}
Here $ (W_{t})_{t\geq 0}$ is a Brownian motion, $m \in \mathbb{R}$ ,
 $(N_{t})_{t\geq 0} $ is a counting Poisson process with intensity $\lambda$ and $( Y_{i})_{i\in \mathbb{N}^{*}}$ 
a sequence
of identically and independent (iid) random variables with distribution $ F_{Y}.$ All these objects are independent.
The process $X $ models the firm value.\\
The default is modeled by the hitting time of level $x>0.$ That means 
\begin{equation}
 \tau_{x} = \inf\{ t > 0: X_{t} \geq x\} .
\end{equation}
We recall from Coutin and Dorobantu \cite{Coutin-Dorobantu} that $\tau_{x} $ has a density $f(.,x)$ with respect to 
Lebesgue measure possibly defective which is defined by 
 \begin{equation}\label{density}
 f(t,x)=\left\{ \begin{array}{ll}
                   \lambda\mathbb{E}(1_{\tau_{x}>t}(1-F_{Y})(x-X_{t}))+ \mathbb{E}(1_{\tau_{x} > T_{N_{t}}}
                   \tilde{f}(t-T_{N_{t}},x-X_{T_{N_{t}}})) & \quad \mbox{ if } t > 0\\
                   \frac{\lambda}{2}(2-F_{Y}(x)-F_{Y}(x_{-}))+\frac{\lambda}{4}(F_{Y}(x)-F_{Y}(x_{-})) & \mbox{ if } t=0.\\ 
                 \end{array} \right. 
\end{equation}
 where $(T_{i}, i \in \mathbb{N}^{*})$ is the sequence of the jump times of the process $N$.

\subsection{Filtering framework}
Instead of observing perfectly the process $ X $, we observe a process $ Q $ defined by 
              $$ Q_{t}=\int_{0}^{t}h(X_{s})ds+ B_{t},\quad t\in \mathbb{R}_{+} $$
where $ h $ is a Borel and bounded function, $ B $ a Brownian motion which is independent of $ W, N, $ and $ (Y_i, i \geq 1) $ .
This is a filtering problem and we introduce the framework as in [5] and [1].\\
Let $ ( \Omega^{Q}, \mathcal{F}^{Q},( \mathcal{F}_{t}^{Q}, t\geq 0), \mathbb{P}^{Q} ) $ (respectively
$ ( \Omega^{W}, \mathcal{F}^{W},( \mathcal{F}_{t}^{W}, t\geq 0), \mathbb{P}^{W} ) $ ) be a measured space on which $ Q $ 
(respectively $ W $ ) is a $\mathbb{R}-$ valued Brownian motion.\\
Let $ ( \Omega^{M}, \mathcal{F}^{M},( \mathcal{F}_{t}^{M}, t\geq 0), \mathbb{P}^{M} ) $ be a measured space on which 
$(Y_{i}, i\in \mathbb{N}^{*})$ is a sequence of i.i.d random variables with distribution function $ F_{Y}$ and $ M $ a Poisson
random measure with intensity $ \Pi(dt, A)= \lambda\int_{A}F_{Y}(dy)dt, \quad \lambda > 0. $ We define 
\begin{eqnarray*}
   \Omega^{\circ}=\Omega^{Q} \times \Omega^{W} \times \Omega^{M} &  & \mathcal{F}=\mathcal{F}^{Q} \otimes \mathcal{F}^{W}
   \otimes \mathcal{F}^{M} \\
   \mathbb{P}^{\circ}= \mathbb{P}^{Q}\otimes \mathbb{P}^{W} \otimes \mathbb{P}^{M} &  & \mathcal{F}_{t}= 
   \mathcal{F}_{t}^{Q} \otimes \mathcal{F}_{t}^{W}\otimes \mathcal{F}_{t}^{M}, t \in \mathbb{R}_{+}.
\end{eqnarray*}
All used filtrations are c\`ad and complete. 
Under $\mathbb{P}^{\circ}, $ we consider the pair processes $(X,Q)$ defined by
 \begin{eqnarray*}
    X_{t} & = & mt +W_{t} +\sum_{i=1}^{N_{t}}Y_{i} \\
    Q_{t} & = &\mbox{ Brownian motion}, t\in \mathbb{R}.
 \end{eqnarray*}
\begin{rems}
\begin{enumerate}
\item $(X,Q)$ has stationary and independent increments under $\mathbb{P}^{0}$,
it is then a $(\mathbb{P}^o,\mathcal{F})-$ Markov process.
 \item The process $(W, Q ) $ is a $\mathbb{R}^{2}-$ valued $ (\mathbb{P}^{\circ},\mathcal{F})-$ Brownian motion .

 \item The compensated  measure $\tilde{M} $ has $ (\mathbb{P}^{\circ},\mathcal{F})-$intensity 
 $ \Pi(dt, A)= \lambda \int_{A}F_{Y}(dy)dt.$ 
\end{enumerate}
\end{rems}
 Since the function $ h $ is bounded the Novikov condition,
$\forall t,~~ \mathbb{E}\left( e^{\frac{1}{2}\int_{0}^{t}h^{2}(X_{s})ds}\right) < \infty ,$ is satisfied and we define 
the following
exponential martingale for the filtration $\mathcal{F}$ by
$$ L_{t}= \exp\left( \int_{0}^{t}h(X_{s})dQ_{s}-\frac{1}{2}\int_{0}^{t}h^{2}(X_{s})ds \right), \quad t \in \mathbb{R}_{+}. $$
For a fixed maturity $ T> 0 $, the process $( L_{t\wedge T}, t\in \mathbb{R}_{+} ) $ is a uniformly integrable 
$ (\mathbb{P}^{\circ}, \mathcal{F} )-$martingale.
 \begin{definition}
  The  probability $ \mathbb{P}, $ called observation probability, is defined as follow 
  $$ \frac{d \mathbb{P}}{d \mathbb{P}^{\circ}}\mid_{\mathcal{F}_{T}}= L_{T}. $$
 \end{definition}
 The probability measures $ \mathbb{P} $ and $\mathbb{P}^{\circ} $ are equivalent. Then using Girsanov theorem, 
  the process $(W, B ) $ is a $\mathbb{R}^{2}-$ valued $ (\mathbb{P},\mathcal{F})-$ Brownian motion.
 The signal $ X $ and the observation
 $ Q $ are represented under $\mathbb{P} $ by
 \begin{eqnarray*}
    X_{t} & = & mt +W_{t} +\sum_{i=1}^{N_{t}}Y_{i}, \\
    Q_{t} & = & \int_{0}^{t}h(X_{s})ds + B_{t}, t\in \mathbb{R}.
 \end{eqnarray*}
 
Then, under $\mathbb{P},$   the process $X$ has stationary and independent increments. It stay a Markov process .
We also note that the law of $X$, so $\tau_x$ under $\mathbb{P}^{0}$ is the same as under $\mathbb{P}.$

\section{Existence of the conditional density}
We introduce the following filtrations:
\begin{align*}
\displaystyle \mathcal{D}_{t} &=\displaystyle \sigma({\bf 1}_{\tau_{x}\leq u}, u\leq t) ,\\
 \displaystyle\mathcal{F}_{t}^{Q} &=\displaystyle \sigma(Q_{u}, u \leq t),\\
 \displaystyle\mathcal{G}_{t} & = \displaystyle \mathcal{F}_{t}^{Q} \vee  \mathcal{D}_{t},~~ \forall t \geq 0.
\end{align*}
The filtration $ (\mathcal{G}_{t})_{t\geq 0}$ models information available on both firm and default occurence for 
investors at time t. It models all available information.
These filtrations are supposed c\`ad and complete.
\begin{proposition} \label{propo-exist-density}
 For all $ t>0,$ on the set $\{\tau_{x}>t\},$ the
 $\mathcal{G}_{t}$ conditional   law of $\tau_{x}$ has the following form
 \begin{equation}
  \bar{f}(r,t,x)dr+ \mathbb{P}(\tau_x =\infty |\mathcal{G}_t)\delta_\infty (dr) 
 \mbox{ and } \mathbb{P}(\tau_x =\infty |\mathcal{G}_t)= 
 {\bf 1}_{\tau_x>t}\mathbb{E}(G(\infty, x-X_t)|\mathcal{G}_t ), 
 \end{equation}
where
\begin{equation*}
 \bar{f}(r,t,x) :=\mathbb{E}[f(r-t,x-X_{t})|\mathcal{G}_{t}].
\end{equation*}
\end{proposition}
\begin{rems}
\label{rk2}
 
We introduce the function $G$ defined as
$G(t,x)=\mathbb{P}(\tau_{x}>t)=\mathbb{P}^o(\tau_{x}>t).$
Refering to \cite{volpi}, for all $x >0$, the passage time $\tau_x $ is finite almost surely if and only if
 $ m+ \mathbb{E}(Y_{1}) \geq 0.$

\end{rems}
\noindent
The next lemma is an auxiliary result. It allows to bound up the density function of the hitting time and later to
interchange the integral and the expectation in the term 
$ \mathbb{E}\left[{\bf 1}_{\tau_{x}>t}\int_{t}^{b}f(r-t,x-X_{t})dr |\mathcal{G}_{t}\right].$

\begin{lemme}\label{lemme0}
 There exists some constants $\tilde{C}$ and $ C $ such that $ \forall t>0, ~~ x>0,$
 \begin{equation}\label{equatlemme0}
  \displaystyle\tilde{f}(t,x) \leq \tilde{C}(\frac{1}{t}+ \frac{1}{\sqrt{t}}) \mbox{ and }
  \displaystyle f(t,x)\leq C(1+  \displaystyle\frac{1}{\sqrt{t}}+\frac{1}{t^{\frac{3}{2}}}).
 \end{equation}
\end{lemme}
\begin{proof}
On one hand, we have $ \forall t >0$:
 \begin{align*}
\displaystyle  \tilde{f}(t,x)&=\displaystyle \frac{x}{\sqrt{2\pi t^{3}}}\exp[-\frac{(x-mt)^{2}}{2t}]\\
                &=\displaystyle\frac{x-mt}{\sqrt{2\pi t^{3}}}\exp[-\frac{(x-mt)^{2}}{2t}]+\frac{mt}{\sqrt{2\pi t^{3}}}
                \exp[-\frac{(x-mt)^{2}}{2t}]\\
                &\displaystyle\leq \left[ \frac{|x-mt|}{\sqrt{2\pi t^{3}}}+\frac{|m|}{\sqrt{2\pi t}}
                \right]\exp[-\frac{(x-mt)^{2}}{2t}]
 \end{align*}
  Let $C_0 = \sup_{u\in \mathbb{R}}ue^{-\frac{u^{2}}{2}}$ with $ u = \frac{|x-mt|}{\sqrt{t}}.$
  It follows that
 \begin{equation*}
  \displaystyle \tilde{f}(t,x) \leq \displaystyle\frac{C_0}{t\sqrt{2 \pi}} + \frac{|m|}{\sqrt{2 \pi t}},~~ t\in \mathbb{R}_{+},
  x \in \mathbb{R}_{+}.
 \end{equation*} 
 The first part of result come from taking  $\tilde{C}= \frac{1}{\sqrt{2 \pi}}(C_0 + |m|)$.\\
On another hand, according to \cite{Coutin-Dorobantu},  the function $f$ defined in (\ref{density}) satisfies
\begin{equation*}
 f(t,x)\leq \lambda + \mathbb{E}(1_{\tau_{x} > T_{N_{t}}} \tilde{f}(t-T_{N_{t}},x-X_{T_{N_{t}}})), \forall t>0,
\end{equation*}
where
\begin{equation*}
 \mathbb{E}(1_{\tau_{x} > T_{N_{t}}} \tilde{f}(t-T_{N_{t}},x-X_{T_{N_{t}}}))\leq 
 \mathbb{E}\left( {\bf 1}_{\{X_{T_{N_t}}>0\}}\tilde{f}(t-T_{N_{t}},
 X_{T_{N_t}})\right),
\end{equation*}
with $ X_{T_{N_t}}= x-mT_{N_{t}}-\!\!\sum_{i=1}^{N_{t}}Y_{i}-\sqrt{T_{N_{t}}}B_{1}.$ 

and $B_{1}$ is a Gaussian random variable independent of $N$, $Y_{i}, i\in \mathbb{N}^{*}. $ According to Lemma 3.1 of the 
appendix of \cite{Coutin-Dorobantu}, we obtain 
\begin{align*}
 f(t,x) &\leq \lambda + \mathbb{E}\left(\!\!{\bf 1}_{X_{T_{N_t}}>0}
 \frac{|X_{T_{N_t}}|}{\sqrt{2\pi(t-T_{N_{t}})^{3}}}
 \exp\left[\!-\frac{(X_{T_{N_t}})^{2}}{2(t-T_{N_{t}})}\right]\right)\\
 \leq &\lambda+\mathbb{E}\left(
 \frac{[X_{T_{N_t}}]_{+}}{\sqrt{2\pi(t-T_{N_{t}})^{3}}}
 \exp\left[-\frac{(X_{T_{N_t}})^{2}}{2(t-T_{N_{t}})}\right]\right).\\
\end{align*}
Applying Lemma  \ref{lemme5} in Appendix to  $t= t-T_{N_{t}}, \sigma=\sqrt{T_{N_{t}}},$ 
 $G=B$ and $\mu=x-mt-\sum Y_i$ 
it follows that
\begin{equation*}
  f(t,x)\leq \lambda + \frac{C_1}{t^{\frac{3}{2}}} +\frac{C_2}{t^{\frac{1}{2}}} +
  \frac{C_3}{t} \mathbb{E}\left(\sqrt{\frac{T_{N_{t}}}{t-T_{N_{t}}}}\right).
\end{equation*}
We take $C= \max\{C_1, C_2, C_3+\lambda \}$ and the proof is completed with Lemma \ref{lemme6} of the Appendix.
\end{proof}
Now, we prove  proposition \ref{propo-exist-density}.
\begin{proof} of proposition \ref{propo-exist-density}:\\

 First note that, since $X$ is a $(\mathcal{F}, \mathbb{P})-$ Markovian process, we have
 \begin{align*}
  \mathbb{E}({\bf 1}_{\tau_{x}=\infty}|\mathcal{G}_{t})&= \mathbb{E}\left(\mathbb{E}({\bf 1}_{\tau_x =\infty}|\mathcal{F}_t)|
  \mathcal{G}_t\right)\\
  &= \mathbb{E}[{\bf 1}_{\tau_x >t}\mathbb{E}^{t}({\bf 1}_{\tau_{x-X_t} =\infty})|
  \mathcal{G}_t]\\
  &= {\bf 1}_{\tau_{x}>t} \mathbb{E}(G(\infty,x-X_t)|\mathcal{G}_t), \mbox{ where }
  \mathbb{E}^{t}(.)=\mathbb{E}(.|\mathcal{F}_t).
  \end{align*}
 The fact that $\tau_x $ is a $ (\mathcal{G}, \mathbb{P})-$ stopping time justifies the last equality.
 
For all $b \geq t$ the $(\mathbb{P},{\cal F})$ Markov property
 of the process $X$  and the fact that on the set $\{\tau_x>t\}$~:  $\tau_x=t+\tau_{x-X_t}\circ\theta_t$  ensure
\begin{align*}
\mathbb{E}({\bf 1}_{a\leq \tau_{x}<b}|\mathcal{G}_{t}) &=\mathbb{E}\left(\mathbb{E}(
{\bf 1}_{a\leq \tau_{x}<b}|\mathcal{F}_{t})|\mathcal{G}_t\right)\\
&= \mathbb{E}\left({\bf 1}_{\tau_x>t}\mathbb{E}^{t}(
{\bf 1}_{a-t\leq \tau_{x-X_t}<b-t})|\mathcal{G}_t\right).
\end{align*}
The $\mathcal{F}_t$- conditional law of $\tau_{x-X_t}$ has the density (possibly defective) $ f(.-t,x-X_t) $, thus 
\begin{equation*}
\mathbb{E}({\bf 1}_{a\leq \tau_{x}<b}|\mathcal{G}_{t})
= \mathbb{E}\left[{\bf 1}_{\tau_{x}>t}\int_{a}^{b}f(r-t,x-X_{t})dr |\mathcal{G}_{t}\right].
  \end{equation*}
By hypothesis, we have $r-t\geq a-t>0.$ It follows from lemma \ref{lemme0} that

\begin{equation*}
 \mathbb{E}\left[{\bf 1}_{\tau_{x}>t}\int_{a}^{b}f(r-t,x-X_{t})dr \right]< \infty.
\end{equation*}
Then, we have forall $b \geq t,$
\begin{equation*}
 \mathbb{E}\left[{\bf 1}_{\tau_{x}>t}\int_{a}^{b}f(r-t,x-X_{t})dr |\mathcal{G}_{t}\right]  =
 \int_{a}^{b}\mathbb{E}\left[{\bf 1}_{\tau_{x}>t}f(r-t,x-X_{t}) |\mathcal{G}_{t}\right]dr ~~ a.s.
\end{equation*}

Now, we show the equality $a.s$ forall $ b\geq t.$ Let $M_1$ and $M_2$ be  the processes defined by 
\begin{equation*}
 M_{1}:b\longmapsto \mathbb{E}\left[{\bf 1}_{\tau_{x}>t}\int_{a}^{b}f(r-t,x-X_{t})dr |\mathcal{G}_{t}\right] \mbox{ and }
 M_{2}:b\longmapsto  \int_{a}^{b}\mathbb{E}\left[{\bf 1}_{\tau_{x}>t}f(r-t,x-X_{t}) |\mathcal{G}_{t}\right]dr.
\end{equation*}
These processes are increasing, then they are 
submartingales with respect to the filtration $\mathcal{\tilde{G}}_{b}=\mathcal{G}_{t}~~ \forall b\geq t.$
Note that  $b\longmapsto \mathbb{E}(M_{1}(b))$ and $b\longmapsto\mathbb{E}(M_{2}(b))$ are too continuous.
Using  Revuz-Yor  Theorem 2.9 p. 61 \cite{revuzyor}, they have same c\`ad-l\`ag modification for all b, meaning that
\begin{equation*}
 \mathbb{E}\left[{\bf 1}_{\tau_{x}>t}\int_{a}^{b}f(r-t,x-X_{t})dr |\mathcal{G}_{t}\right]=
 \int_{a}^{b}\mathbb{E}\left[{\bf 1}_{\tau_{x}>t}f(r-t,x-X_{t}) |\mathcal{G}_{t}\right]dr ~~~~a.s. \forall b.
\end{equation*}
We conclude that, almost surely, for all $ b> a > t,$
 \begin{equation*}
  {\bf 1}_{\tau_{x}>t}\mathbb{E}({\bf 1}_{a<\tau_{x}\leq b}|\mathcal{G}_{t})= {\bf 1}_{\tau_{x}>t}
  \int_{a}^{b}\mathbb{E}\left[{\bf 1}_{\tau_{x}>t}f(r-t,x-X_{t})|\mathcal{G}_{t} \right]dr.
 \end{equation*}
 
Taking $ a =t + \frac{1}{n}$, letting $n$ going to infinity  and using
monotone Lebesgue Theorem yield that,
 $\mathbb{P}-a.s~~ \forall ~~ b, $
\begin{equation*}
 \mathbb{E}({\bf 1}_{t<\tau_{x}\leq b}|\mathcal{G}_{t}) = 
  \int_{t}^{b}\mathbb{E}\left[{\bf 1}_{\tau_{x}>t}f(r-t,x-X_{t})
 |\mathcal{G}_{t}\right]dr.
\end{equation*}
\end{proof}

 \section{Mixed filtering-Integro-differential equation for
\\
 conditional density}
In this section, we give one of our main results.  Indeed, we show that the conditional law of the hitting time $\tau_x $
given the filtration $(\mathcal{G}_{t})_{t\geq 0} $ satisfies some stochastic integrodifferential equation.
\begin{theo}
\label{th1}
 Let $t>0$ be a real number. For any $r>t,$ on the set $\{ \tau_{x}>t\},$ the conditional
density of $\tau_{x}$ given $\mathcal{G}_{t}$ satisfies the stochastic  integrodifferential equation
\begin{align*} \displaystyle
 \bar{f}(r,t,x)  &= \displaystyle \frac {f(r,x)}{\mathbb{P}(\tau_{x}>t)} 
 + \displaystyle\int_{0}^{t}\Pi^{1}(h)(t,r,u) dQ_{u}\\
   \displaystyle   
  & -\displaystyle\int_{0}^{t} \frac{\bar{f}(r,u,x)}{\mathbb{E}({\bf 1}_{\tau_x >u}G(t-u,x-X_u)|\mathcal{G}_u)}
   \Pi(t,u)(h)dQ_{u}\\ 
   \displaystyle
   &+\displaystyle\int_{0}^{t} \frac{\bar{f}(r,u,x)}{\mathbb{E}({\bf 1}_{\tau_x >u}G(t-u,x-X_u)|\mathcal{G}_u)} 
   [\Pi(t,u)(h)]^{2}du\\ 
   \displaystyle
   &-\displaystyle\int_{0}^{t}\Pi^{1}(h)(t,r,u) \Pi(h)(t,u)du .
  \end{align*}
  where
  \begin{align*}
   \Pi^{1} (t,r,u)(\Phi)&= \frac{\mathbb{E}({\bf 1}_{\tau_{x}>u}\Phi(X_{u})f(r-u,x-X_{u})|\mathcal{G}_{u})}
   {\mathbb{E}({\bf 1}_{\tau_{x}>u}G(t-u,x-X_{u})|\mathcal{G}_{u})},\\
   \Pi(t,u)(\Phi)& = \frac{\mathbb{E}({\bf 1}_{\tau_{x}>u}\Phi(X_{u})G(t-u,x-X_{u})|
   \mathcal{G}_{u})} {[\mathbb{E}({\bf 1}_{\tau_{x}>u}G(t-u,x-X_{u})|\mathcal{G}_{u})]^{2}}
  \end{align*}

and $G$ is defined in Remark\ref{rk2}
\end{theo}

The next lemma is inspired of Jeanblanc \cite{jeanblanc-rut} and Dorobantu \cite{Dorobantu}.
 \begin{lemme}
 \label{lemme1}
 For all $ t \in \mathbb{R}_{+}, $ for all $a$ and $b$ such that $t< a <b,$  for all $Y\in L^1({\cal F}_b,
 \mathbb{P})$
\begin{equation}
\label{equat1}
\mathbb{E}({\bf 1}_{\tau_{x}>t}|\mathcal{F}_{t}^{Q})> 0,~~ 
\mathbb{E}( Y{\bf 1}_{t<\tau_{x}}|\mathcal{G}_{t}) 
 = {\bf 1}_{\tau_{x}>t}
 \frac{\mathbb{E}^{0}(L_{b}Y{\bf 1}_{t<\tau_{x}}|\mathcal{F}_{t}^{Q})}
 {\mathbb{E}^{0}({\bf 1}_{\tau_{x}>t}L_{t}|\mathcal{F}_{t}^{Q})}. 
 \end{equation}
 For instance with $Y={\bf 1}_{a<\tau_{x}<b},$ we get
 $$
\mathbb{E}( {\bf 1}_{a<\tau_{x}<b}|\mathcal{G}_{t}) 
 = {\bf 1}_{\tau_{x}>t}
 \frac{\mathbb{E}^{0}(L_{b}{\bf 1}_{a<\tau_{x}<b}|\mathcal{F}_{t}^{Q})}
 {\mathbb{E}^{0}({\bf 1}_{\tau_{x}>t}L_{t}|\mathcal{F}_{t}^{Q})}.
$$
\end{lemme}
\begin{proof}
 Assume that there exists $t_{0} $ such that 
$ \mathbb{P}(\tau_{x}>t_{0})=0.$ Then for all $t\geq t_{0},\\
\quad \mathbb{P}(\tau_{x}\leq t_{0})=1.$ It follows that the density
function of $\tau_{x} $ $f,$ defined in (\ref{density}), is the zero function on $[ t_{0}, +\infty[.$ This means that 
$\forall t\in[t_0,\infty[,$
\begin{equation*}
~f(t,x)= \lambda\mathbb{E}(1_{\tau_{x}>t}(1-F_{Y})(x-X_{t}))+ \mathbb{E}(1_{\tau_{x} > T_{N_{t}}}
                   \tilde{f}(t-T_{N_{t}},x-X_{T_{N_{t}}}))=0 \quad \mathbb{P}-a.s.
\end{equation*}
Then, $ \mathbb{P}(\tau_{x}\leq t)=1$ implies that $  \mathbb{E}(1_{\tau_{x}>t}(1-F_{Y})(x-X_{t}))=0.$ \\ Thus 
$\mathbb{E}(1_{\tau_{x} > T_{N_{t}}} \tilde{f}(t-T_{N_{t}},x-X_{T_{N_{t}}}))=0 .$ But we have $t-T_{N_{t}}>0\quad \mathbb{P}
-a.s$ and on the set $\{\tau_{x}>T_{N_{t}}\}, \quad x-X_{T_{N_{t}}}>0.$ Therefore, $\tilde{f}(t-T_{N_{t}},x-X_{T_{N_{t}}})>0$
for all $t \geq t_{0}.$ Hence, we obtain ${\bf 1}_{\tau_{x}>T_{N_{t}}}=0, \forall t \geq t_{0}$
what is not possible. Indeed, 
\begin{equation*}
 {\bf 1}_{\tau_x > T_{N_{t}}}=0 \Longleftrightarrow \sum_{n\geq 0} {\bf 1}_{\tau_x > T_n}{\bf 1}_{N_t = n} =0
\end{equation*}
That means for all $n \in \mathbb{N}, \mathbb{P}(T_n < t< T_{n+1}, \tau_x > T_n) = 0.$  In particular, for $n = 0,$
\begin{equation*}
 \mathbb{P}(T_1 > t, \tilde{\tau}_x > 0)= \mathbb{P}(\tilde{\tau}_x > 0)\mathbb{P}(T_1 > t) = e^{\lambda t} \neq 0.
\end{equation*}
 Thus for any $t,$ $ \mathbb{P}(\tau_{x}>t)>0$ and  $ \mathbb{E}({\bf 1}_{\tau_{x}>t}|\mathcal{F}_{t}^{Q})> 0.$
 \\
 
On the set $\{\tau_{x}> t\}$, any $\mathcal{G}_{t}-$ mesurable random variable coincides with some $\mathcal{F}_{t}^{Q}-$
measurable random variable (cf. Jeanblanc and Rutkovski
 \cite{jeanblanc-rut} p. 18).  Then for all $Y\in L^1({\cal F}_b,\mathbb{P})$,  there exists a
$   \mathcal{F}_{t}^{Q}-$ measurable random variable $ Z $ such that
\begin{equation*}
 \mathbb{E}({\bf 1}_{\tau_{x}>t} Y|\mathcal{G}_{t}) = {\bf 1}_{\tau_{x}>t}Z.
\end{equation*}
Taking the conditional expectation with respect to $\mathcal{F}_{t}^{Q}, $ we get
\begin{equation*}
  \mathbb{E}({\bf 1}_{\tau_{x}>t} Y|\mathcal{F}_{t}^{Q}) = Z \mathbb{E}(\tau_{x}>t|\mathcal{F}_{t}^{Q}).
\end{equation*}
This implies that 
\begin{equation*}
 \mathbb{E}({\bf 1}_{\tau_{x}>t} Y|\mathcal{G}_{t})=     {\bf 1}_{\tau_{x}>t} \frac{\mathbb{E}(Y {\bf 1}_{\tau_{x}>t}
 |\mathcal{F}_{t}^{Q})}{\mathbb{E}({\bf 1}_{\tau_{x}>t}|\mathcal{F}_{t}^{Q})}.
 \end{equation*}
 Using Kallianpur-Striebel formula (see Pardoux  \cite{Pardoux})  and
$\mathbb{E}^{0}(L_{b}|\mathcal{F}_{t}^{Q})=L_t$ we obtain 
\begin{equation*}
\label{equat2}
\mathbb{E}(Y {\bf 1}_{\tau_{x}>t}|\mathcal{G}_{t}) =
 {\bf 1}_{\tau_{x}>t}
 \frac{\mathbb{E}^{0}(L_{b}{\bf 1}_{\tau_{x}>t} Y|\mathcal{F}_{t}^{Q})}
 {\mathbb{E}^{0}({\bf 1}_{\tau_{x}>t}L_{t}|\mathcal{F}_{t}^{Q})} .
\end{equation*}
\end{proof}
The following is in \cite{Coutin}.
\begin{lemme}
\label{lemme2}
The family 
\begin{equation*}
\mathcal{S}_{t} = \left\{ S_{t} = \exp\left(\int_{0}^{t}\rho_{s}dQ_{s} - \frac{1}{2}\int_{0}^{t}\rho_{s}^{2}ds\right),
\rho \in L^{2} ([0,T], \mathbb{R})\right\}
\end{equation*}
is total in $L^{2}(\Omega, \mathcal{F}_{t}^{Q}, \mathbb{P}^{0}).$
\end{lemme}
\begin{lemme}
\label{lemme3}
 Let $ \{U_{t}, t\geq 0\} $ be an 
 $\mathcal{F}^{W}\otimes\mathcal{F}^M-$progressively measurable process such that for all $ t\geq 0, $  we have 
 \begin{equation*}\displaystyle
  \mathbb{E}^{0}\left[\int_{0}^{t}U_{s}^{2}ds \right] < + \infty.
 \end{equation*}
Then
\begin{equation}\displaystyle
     \mathbb{E}^{0}\left[\int_{0}^{t}U_{s}dQ_{s}|\mathcal{F}_t^{W}\otimes\mathcal{F}_t^M\right] = 0.
\end{equation}
\end{lemme}
\begin{proof}
As in Lemma \ref{lemme2}, the family 
\begin{equation*}
 \mathcal{R}_{t} =\left\{ r_{t} = \mathcal{E}\left[\int_{0}^{t}\gamma_{s}dW_{s} +\int_{0}^{t}\int_{A}(e^{\beta_{s}(x)}-1)
 \tilde{M}(dsdx)\right], \gamma \in L^{2}([0,T],\mathbb{R}]), \beta \in L^{\infty}([0,T]\times A, \mathbb{R})\right\}
\end{equation*}
is total in $ L^{2}( \Omega, \mathcal{F}^{W}\otimes\mathcal{F}^M, \mathbb{P}^{0}),$ where $\tilde{M} $ is a compensated Poisson random measure
on 
$\mathbb{R}\times \mathbb{R}$ and $ A \subset \mathbb{R} $ is a Borel set. 
Therefore, since $ r_{t} = 1 + \int_{0}^{t}r_{s}\gamma_{s}dW_{s}+\int_{0}^{t}\int_{A}r_{s}(e^{\beta_{s}(x)}-1)
\tilde{N}(dsdx),$ by It\^o's formula, we have
\begin{align*}
 \mathbb{E}^{0}\left( r_{t}\mathbb{E}^{0}\left[\int_{0}^{t}U_{s}dQ_{s}|\mathcal{F}_t^{W}\otimes\mathcal{F}_t^M\right]\right)&=
 \mathbb{E}^{0}\left[r_{t}\int_{0}^{t}U_{s}dQ_{s}\right] \\
 & = \mathbb{E}^{0}\left[\int_{0}^{t}r_{s}\gamma_{s}U_{s}d<W,Q>_{s}\right]\\
 &+ \mathbb{E}^{0}\left[\int_{0}^{t}U_{s}\int_{A}r_{s}(e^{\beta_{s}(x)}-1)d<\tilde{M},Q>_{s}\right]= 0 .
\end{align*}
The equality is obtained from the fact that  $<Q,W>=<Q,\tilde{M}>=0$ by independence. 
\end{proof}
\begin{proposition}\label{propo2}
For any bounded function $\varphi$ such that $\varphi(\tau_x)$ is ${\cal F}^X_T$-measurable,
\begin{equation}\label{equat3}
 \mathbb{E}^{0}(\varphi(\tau_{x}){\bf 1}_{\tau_{x>t}}L_{T}|\mathcal{F}_{t}^{Q})= 
 \mathbb{E}^{0}[\varphi(\tau_{x}){\bf 1}_{\tau_{x>t}}] + 
  \int_{0}^{t}\mathbb{E}^{0}[{\bf 1}_{\tau_{x>u}}L_{u}h(X_{u})
 \varphi(\tau_{x-X_{u}})|\mathcal{F}_{u}^{Q}]dQ_{u}.
\end{equation}
\end{proposition}
\begin{proof}
Let $ S_{t} \in \mathcal{S}_{t}. $
Lemma \ref{lemme7} applied to    $Y=~\varphi(\tau_x){\bf 1}_{\tau_{x}>t}$  
which belongs to $L^\infty(\Omega,\mathbb{P}^{0},{\cal F}^X_T)$
implies 
\begin{equation*}
  \mathbb{E}^{0}(\varphi(\tau_{x}){\bf 1}_{\tau_{x>t}}L_{T}|\mathcal{F}_{t}^{Q})=  
  \mathbb{E}^{0}[\varphi(\tau_{x}){\bf 1}_{\tau_{x>t}}] +
\mathbb{E}^{0}\left(\int_{0}^{t} \mathbb{E}^{0}(\varphi(\tau_x){\bf 1}_{\tau_{x}>t}|\mathcal{F}_{u})
L_{u}S_{u}\rho_{u}h(X_{u})du \right).
\end{equation*}
The $(\mathbb{P}^{0}, \mathcal{F})-$ Markov property of $X$ permits us to write
\begin{align*}
 \mathbb{E}^{0}(\varphi(\tau_{x}){\bf 1}_{\tau_{x>t}}L_{T}|\mathcal{F}_{t}^{Q})=  &
  \mathbb{E}^{0}[\varphi(\tau_{x}){\bf 1}_{\tau_{x>t}}] +\\
 & \mathbb{E}^{0}\left(\int_{0}^{t}L_{u}S_{u}\rho_{u}h(X_{u}){\bf 1}_{\tau_{x}> u}
\mathbb{E}^{0,u}(\varphi(\tau_{x-X_u}){\bf 1}_{\tau_{x-X_u >t-u}})du \right)\\
\mbox{ where } \mathbb{E}^{0,u} (.)= \mathbb{E}^{0}(.|\mathcal{F}_u).
\end{align*}
Conditioning by $ \mathcal{F}_{u}^{Q}$  under the time integral,
it follows that
\begin{align*}
\mathbb{E}^{0}(\varphi(\tau_{x}){\bf 1}_{\tau_{x>t}}L_{T}|\mathcal{F}_{t}^{Q})=  &
  \mathbb{E}^{0}[\varphi(\tau_{x}){\bf 1}_{\tau_{x>t}}] +\\ 
& \mathbb{E}^{0}\left(\int_{0}^{t}S_{u}\rho_{u}\mathbb{E}^{0}(L_u h(X_{u}){\bf 1}_{\tau_{x}> u}
\varphi(\tau_{x-X_u})|\mathcal{F}_u^Q)du \right).
\end{align*}
Conversely compute the expectation of the product of $S_t=1+\int_0^tS_u\rho_udQ_u$ by right part of (\ref{equat3}):
\begin{align*}
&\mathbb{E}^{0}\left[S_t(\mathbb{E}^{0}[\varphi(\tau_{x}){\bf 1}_{\tau_{x>t}}] +  
  \int_{0}^{t}\mathbb{E}^{0}(L_u h(X_{u}){\bf 1}_{\tau_{x}> u}
\varphi(\tau_{x-X_u})|\mathcal{F}_u^Q)dQ{u})\right]=\\
  & \mathbb{E}^{0}[\varphi(\tau_{x}){\bf 1}_{\tau_{x>t}}] +
 \mathbb{E}^{0}\left(\int_{0}^{t}S_{u}\rho_{u}\mathbb{E}^{0}(L_u h(X_{u}){\bf 1}_{\tau_{x}> u}
\varphi(\tau_{x-X_u})|\mathcal{F}_u^Q)du \right)
\end{align*}
Since $\mathcal{S} $ is dense in $ L^{2}(\Omega, \mathcal{F}^{Q}, \mathbb{P}^{0}),$ 
 \begin{equation*}
\mathbb{E}^{0}(\varphi(\tau_{x}){\bf 1}_{\tau_{x>t}}L_{T}|\mathcal{F}_{t}^{Q})= 
 \mathbb{E}^{0}[\varphi(\tau_{x}){\bf 1}_{\tau_{x>t}}] + 
  \int_{0}^{t}\mathbb{E}^{0}[{\bf 1}_{\tau_{x>u}}L_{u}h(X_{u})
 \varphi(\tau_{x-X_{u}})|\mathcal{F}_{u}^{Q}]dQ_{u}.
 \end{equation*}
 \end{proof}
 By this proposition, we etablish two corollaries which give a representation more accessible of the processes
 $ t \longmapsto \mathbb{E}^{0}({L_{b}\bf 1}_{a<\tau_{x}<b}|\mathcal{F}_{t}^{Q})$ and 
 $ t \longmapsto \mathbb{E}^{0}({\bf 1}_{\tau_{x}>T}L_{T}|\mathcal{F}_{t}^{Q}) $
 
 \begin{corollary}\label{coro1}
 For all $ t< a <b, $ we have $ \mathbb{P}^{0}-a.s $
 \begin{align}
 \label{equat4}
  \mathbb{E}^{0}({L_{b}\bf 1}_{a<\tau_{x}<b}|\mathcal{F}_{t}^{Q}) &=  \mathbb{P}^{0}(a<\tau_{x}<b) +\\ \nonumber \displaystyle
 & \int_{0}^{t}\mathbb{E}^{0}({\bf 1}_{\tau_{x}>u}L_{u}h(X_{u})[G(a-u,x-\!X_{u})-G(b-u,x-\!X_{u})]|
  \mathcal{F}_{u}^{Q})dQ_{u}.
 \end{align}
\end{corollary}
 The next is too a particular case of proposition (\ref{propo2}).
\begin{corollary} \label{coro2}
For $ t< T , $
 \begin{equation}
 \label{equat5} 
   \mathbb{E}^{0}({\bf 1}_{\tau_{x}>T}L_{T}|\mathcal{F}_{t}^{Q})= \mathbb{P}^{0}(\tau_{x}>T) + 
  \int_{0}^{t}\mathbb{E}^{0}({\bf 1}_{\tau_{x>u}}L_{u}h(X_{u})G(T-u,x-X_{u})|\mathcal{F}_{u}^{Q})dQ_{u}.
 \end{equation}
 \end{corollary}

\begin{proposition}
\label{propo3}
  For any $ 0 < t< a < b,$ we have on the set $\{\tau_{x}>t\},$
  \begin{align} \label{equat6}
   \bar{\Gamma}_t & = \frac{\mathbb{P}^{0}(a<\tau_{x}<b)}
   {\mathbb{P}^{0}(\tau_{x}>t)}  + \int_{0}^{t} \sigma^{1} (\Phi)(t,u) dQ_{u}\\ \nonumber
   &-\int_{0}^{t}  \bar{\Gamma}_u \sigma(\Phi)(t,u))^{2}du
   +\int_{0}^{t} \bar{\Gamma}_u\sigma(\Phi)(t,u)dQ_u
   \\ \nonumber
   &-\int_{0}^{t} \sigma^{1}(\Phi)(t,u) \sigma(\Phi)(t,u) du
   . \nonumber
  \end{align}
  where
  \begin{align*}
   \bar{\Gamma}_t &= \mathbb{E}({\bf 1}_{a<\tau_x<b}|\mathcal{G}_t),\\
   \sigma^{1}(\Phi)(t,u)&= \frac{\mathbb{E}^{0}({\bf 1}_{\tau_{x}>u}L_{u}\Phi(X_{u})[G(a-u,x-X_{u})-
   G(b-u,x-X_{u})]|\mathcal{F}_{u}^{Q})} {\mathbb{E}^{0}({\bf 1}_{\tau_{x}>t}L_{u}|\mathcal{F}_{u}^{Q})} ,\\
   \sigma(\Phi)(t,u)&= \frac{\mathbb{E}^{0}({\bf 1}_{\tau_{x}>u}L_{u}\Phi(X_{u})G(t-u,x-X_{u})|\mathcal{F}_{u}^{Q})}
   {\mathbb{E}^{0}({\bf 1}_{\tau_{x}>t}L_{u}|\mathcal{F}_{u}^{Q})}.
  \end{align*}

 \end{proposition}  
 
 \begin{proof}
  We first apply 
It\^o's formula to $\displaystyle \frac{\mathbb{E}^{0}({L_{b}\bf 1}_{a<\tau_{x}<b}|\mathcal{F}_{t}^{Q})}
 {\mathbb{E}^{0}({\bf 1}_{\tau_{x}>T}L_{T}|\mathcal{F}_{t}^{Q})} .$  Second, we take the limit when $T$
 goes to $t.$
 But Lemma \ref{lemme7} of Appendix ensures that  
 $$\displaystyle \frac{\mathbb{E}^{0}({L_{b}\bf 1}_{a<\tau_{x}<b}|\mathcal{F}_{t}^{Q})}
 {\mathbb{E}^{0}({\bf 1}_{\tau_{x}>T}L_{T}|\mathcal{F}_{t}^{Q})} =
 \displaystyle \frac{\mathbb{E}^{0}({L_{t}\bf 1}_{a<\tau_{x}<b}|\mathcal{F}_{t}^{Q})}
 {\mathbb{E}^{0}({\bf 1}_{\tau_{x}>T}L_{t}|\mathcal{F}_{t}^{Q})} .$$
 Therefore, we let two processes satisfying the stochastic equations respectively
 (\ref{equat3}) and (\ref{equat4}):
 \begin{equation*}
  X_t = \mathbb{E}^{0}({L_{t}\bf 1}_{a<\tau_{x}<b}|\mathcal{F}_{t}^{Q}),~~ Y_t = 
  \mathbb{E}^{0}({\bf 1}_{\tau_{x}>T}L_{t}|\mathcal{F}_{t}^{Q}) \mbox{ and } f(x,y) = \frac{x}{y}
 \end{equation*} 
 The It\^o's formula applied to $ f(X_., Y_.)$ from $0$ to $t$ gives us 
 \begin{align*}
  \frac{\mathbb{E}^{0}({L_{b}\bf 1}_{a<\tau_{x}<b}|\mathcal{F}_{t}^{Q})}
 {\mathbb{E}^{0}({\bf 1}_{\tau_{x}>T}L_{T}|\mathcal{F}_{t}^{Q})}& = \frac{\mathbb{P}^{0}(a<\tau_{x}<b)}
   {\mathbb{P}^{0}(\tau_{x}>T)} \\ \displaystyle
   &+ \int_{0}^{t}
   \frac{\mathbb{E}^{0}({\bf 1}_{\tau_{x}>u}L_{u}h(X_{u})[G(a-u,x-X_{u})-G(b-u,x-X_{u})]|\mathcal{F}_{u}^{Q})}
   {\mathbb{E}^{0}({\bf 1}_{\tau_{x}>T}L_{u}|\mathcal{F}_{u}^{Q})}dQ_{u}\\ \displaystyle
   &-\int_{0}^{t}
   \frac{\mathbb{E}^{0}(L_{u}{\bf 1}_{a<\tau_{x}<b}|\mathcal{F}_{u}^{Q})
   \mathbb{E}^{0}({\bf 1}_{\tau_{x}>u}L_{u}h(X_{u})G(T-u,x-X_{u})|
   \mathcal{F}_{u}^{Q})} {[\mathbb{E}^{0}({\bf 1}_{\tau_{x}>T}L_{u}|\mathcal{F}_{u}^{Q})]^{2}}dQ_{u}\\ \displaystyle
   &+\int_{0}^{t}
   \frac{\mathbb{E}^{0}(L_{u}{\bf 1}_{a<\tau_{x}<b}|\mathcal{F}_{u}^{Q})
   [\mathbb{E}^{0}({\bf 1}_{\tau_{x}>u}L_{u}h(X_{u})G(T-u,x-X_{u})|
   \mathcal{F}_{u}^{Q})]^{2}} {[\mathbb{E}^{0}({\bf 1}_{\tau_{x}>T}L_{u}|\mathcal{F}_{u}^{Q})]^{3}}du\\ \displaystyle
   &-\int_{0}^{t}
  \mathbb{E}^{0}({\bf 1}_{\tau_{x}>u}L_{u}h(X_{u})[G(a-u,x-X_{u})-G(b-u,x-X_{u})]|\mathcal{F}_{u}^{Q})\\
  \displaystyle
  & \times \frac{\mathbb{E}^{0}({\bf 1}_{\tau_{x}>u}L_{u}h(X_{u})G(T-u,x-X_{u})|\mathcal{F}_{u}^{Q})}
   {[\mathbb{E}^{0}({\bf 1}_{\tau_{x}>T}L_{u}|\mathcal{F}_{u}^{Q})]^{2}}du . \\ 
 \end{align*}

 Now, we let $T$ goes to $t$. For this end,  we will use stochastic Fubini's theorem, so need uniform majorations
 when $T$ satisfies $t<T\leq t+1.$ We start 
 to show that
 \begin{equation*}
  \mathbb{E}^{0}\left( \int_{0}^{t}\left[ \frac{(Z_{u}^{t})^{i}-(Z_{u}^{t+1})^{i}}{(Z_{u}^{t}Z_{u}^{t+1})^{i}}\right]^{2+
  \varepsilon}du\right) <\infty ,
 \end{equation*}
where $ Z_{u}^{t} = \mathbb{E}^{0}({\bf 1}_{\tau_{x}>t}L_{u}|\mathcal{F}_{u}^{Q})$ and $ i \in \{ 1, 2, 3\}. $ But, since
  $ t\longmapsto Z_{u}^{t} $ is non increasing
so  $ | (Z_{u}^{t})^{i}-(Z_{u}^{t+1})^{i}|=(Z_{u}^{t})^{i}-(Z_{u}^{t+1})^{i} \leq (Z_{u}^{t})^{i}$ and this
 leads us to write using twice Jensen inequality with ($\phi(x)=\frac{1}{x}$ and $\phi(x)=x^{i(2+\varepsilon)}$):
\begin{align*}
  \mathbb{E}^{0}\left( \int_{0}^{t} \left[ \frac{(Z_{u}^{t})^{i}-(Z_{u}^{t+1})^{i}}{(Z_{u}^{t}Z_{u}^{t+1})^{i}}\right]^{2+
  \varepsilon}du\right) &\leq \mathbb{E}^{0}\left( \int_{0}^{t}\frac{du}{(Z_{u}^{t+1})^{i(2+\varepsilon)}}\right)\\
  & \leq \mathbb{E}^{0}\left(\int_{0}^{t}\left[\mathbb{E}^{0}\left(\frac{1}{L_{u}{\bf 1}_{\tau_{x}>t+1}}|\mathcal{F}_{u}^{Q}
  \right)\right]^{i(2+\varepsilon)}du\right)\\
  & \leq \mathbb{E}^{0}\left(\int_{0}^{t}\mathbb{E}^{0}\left[\left(\frac{1}{L_{u}{\bf 1}_{\tau_{x}>t+1}}
  \right)^{i(2+\varepsilon)} |\mathcal{F}_{u}^{Q} \right]du\right),
   \end{align*} 
   the last majoration being equal to $\int_{0}^{t}\mathbb{E}^{0}\left[\left(\frac{1}{L_{u}
{\bf 1}_{\tau_{x}>t+1}} \right)^{i(2+\varepsilon)} \right]du < + \infty.$

 We will use the fact that $G(T-u,x-X_{u})$ goes to $ G(t-u,x-X_{u})$ when $T$ goes to $t$. 
 We could use the  stochastic Lebesgue dominated  convergence theorem since
 \begin{equation*}
  \int_{0}^{t} \left[ \frac{(Z_{u}^{t})^{i}-(Z_{u}^{T})^{i}}{(Z_{u}^{t}Z_{u}^{T})^{i}}\right]^{2+ \varepsilon}du \leq 
  \int_{0}^{t} \left[ \frac{(Z_{u}^{t})^{i}-(Z_{u}^{t+1})^{i}}{(Z_{u}^{t}Z_{u}^{t+1})^{i}}\right]^{2+ \varepsilon}du.
 \end{equation*}
 This added to ordinary dominated  Lebesgue theorem leads to
  \begin{align*}
  \lim_{T\rightarrow t}\frac{\mathbb{E}^{0}({L_{t}\bf 1}_{a<\tau_{x}<b}|\mathcal{F}_{t}^{Q})}
 {\mathbb{E}^{0}({\bf 1}_{\tau_{x}>T}L_{t}|\mathcal{F}_{t}^{Q})}& = \frac{\mathbb{P}^{0}(a<\tau_{x}<b)}
   {\mathbb{P}^{0}(\tau_{x}>t)} \\ \displaystyle
   &+ \int_{0}^{t}
   \frac{\mathbb{E}^{0}({\bf 1}_{\tau_{x}>u}L_{u}h(X_{u})[G(a-u,x-X_{u})-G(b-u,x-X_{u})]|\mathcal{F}_{u}^{Q})}
   {\mathbb{E}^{0}({\bf 1}_{\tau_{x}>t}L_{u}|\mathcal{F}_{u}^{Q})}dQ_{u}\\ \displaystyle
   &-\int_{0}^{t}
   \frac{\mathbb{E}^{0}(L_{u}{\bf 1}_{a<\tau_{x}<b}|\mathcal{F}_{u}^{Q})
   \mathbb{E}^{0}({\bf 1}_{\tau_{x}>u}L_{u}h(X_{u})G(t-u,x-X_{u})|
   \mathcal{F}_{u}^{Q})} {[\mathbb{E}^{0}({\bf 1}_{\tau_{x}>t}L_{u}|\mathcal{F}_{u}^{Q})]^{2}}dQ_{u}\\ \displaystyle
   &+\int_{0}^{t}
   \frac{\mathbb{E}^{0}(L_{u}{\bf 1}_{a<\tau_{x}<b}|\mathcal{F}_{u}^{Q})
   [\mathbb{E}^{0}({\bf 1}_{\tau_{x}>u}L_{u}h(X_{u})G(t-u,x-X_{u})|
   \mathcal{F}_{u}^{Q})]^{2}} {[\mathbb{E}^{0}({\bf 1}_{\tau_{x}>t}L_{u}|\mathcal{F}_{u}^{Q})]^{3}}du\\ \displaystyle
   &-\int_{0}^{t}
  \mathbb{E}^{0}({\bf 1}_{\tau_{x}>u}L_{u}h(X_{u})[G(a-u,x-X_{u})-G(b-u,x-X_{u})]|\mathcal{F}_{u}^{Q})\\
  \displaystyle
  & \times \frac{\mathbb{E}^{0}({\bf 1}_{\tau_{x}>u}L_{u}h(X_{u})G(t-u,x-X_{u})|\mathcal{F}_{u}^{Q})}
   {[\mathbb{E}^{0}({\bf 1}_{\tau_{x}>t}L_{u}|\mathcal{F}_{u}^{Q})]^{2}}du . \\ 
 \end{align*}
 
  \begin{proof} of  Theorem \ref{th1}
 
Let us now find a mixed filtering-integro-differential equation satisfied by the conditional probability density 
process defined from the representation
\begin{equation}
 \mathbb{E}({\bf 1}_{a<\tau_{x}<b}{|\mathcal{G}_{t}})=\int_{a}^{b}\bar{f}(r,t,x)dr \mbox{ for some } a>t.
\end{equation}
Applying Lemma \ref{lemme1}, it follows that 
\begin{equation*}
 \mathbb{E}({\bf 1}_{a<\tau_{x}<b}|\mathcal{G}_{t}) 
 = {\bf 1}_{\tau_{x}>t}
 \frac{\mathbb{E}^{0}(L_{b}{\bf 1}_{a<\tau_{x}<b}|\mathcal{F}_{t}^{Q})}
 {\mathbb{E}^{0}({\bf 1}_{\tau_{x}>t}L_{t}|\mathcal{F}_{t}^{Q})}.
\end{equation*}
By previous corollary \ref{coro2},  we show that 
\begin{equation*}
 \mathbb{E}^{0}({\bf 1}_{\tau_{x}>t}L_{t}|\mathcal{F}_{t}^{Q}) = \mathbb{P}^{0}(\tau_{x}>t) + 
 \int_{0}^{t}\int_{t}^{+\infty}\mathbb{E}^{0}({\bf 1}_{\tau_{x}>u}L_{u}h(X_{u})f(r-u,x-X_{u})|\mathcal{F}_{u}^{Q})drdQ_{u}.
\end{equation*}
But, since the  condition $\int_{0}^{t}\mathbb{E}^{0}(f^{2}(t-u,x-X_{u})))du < \infty $ is not necessarily satisfied,
we are not able 
to prove that $ \mathbb{E}^{0}({\bf 1}_{\tau_{x}>t}L_{t}|\mathcal{F}_{t}^{Q}) $ is a semimartingale
(e.g. see Protter's Theorem 65 \cite{pprotter}). This leads us to consider
for $ t\leq  T <t+1,$ the expression $  \mathbb{E}^{0}({\bf 1}_{\tau_{x}>T}L_{T}|\mathcal{F}_{t}^{Q}). $
In proposition \ref{propo3} we dealed with the term 
$\frac{\mathbb{E}^{0}(L_{b}{\bf 1}_{a<\tau_{x}<b}|\mathcal{F}_{t}^{Q})}
 {\mathbb{E}^{0}({\bf 1}_{\tau_{x}>T}L_{t}|\mathcal{F}_{t}^{Q})}$  and taked the limit when $T$ goes to $t$ which leads us
 to equation (\ref{equat6}).\\
 Now, we recall the $(\mathbb{P}^{0}, \mathcal{F})-$ Markov property of $X$ at point $u$ 
 and the fact that $\mathcal{F}^{Q} \subset \mathcal{F}$  justify
 \begin{align*}
  \mathbb{E}^{0} \left( L_u {\bf 1}_{a<\tau_x<b}[\mathcal{F}_u^Q\right) & =
  \mathbb{E}\left( L_u {\bf 1}_{\tau_x >u} \mathbb{E}^{0,u}({\bf 1}_{a-u<\tau_{x-X_u}<b-u})|\mathcal{F}_u^Q\right)\\
  & = \mathbb{E}\left( L_u {\bf 1}_{\tau_x >u} \int_a^b f(r-u,x-X_u)dr|\mathcal{F}_u^Q\right)\\ 
  &= \int_a^b \mathbb{E}\left( L_u {\bf 1}_{\tau_x >u}  f(r-u,x-X_u)|\mathcal{F}_u^Q\right)dr.
 \end{align*}
Then, the equation (\ref{equat6}) in the proposition \ref{propo3}  can be rewrite as

\begin{align*} 
   \mathbb{E}({\bf 1}_{a < \tau_{x} < b} | \mathcal{G}_{t}) & = \frac{1}
   {\mathbb{P}^{0}(\tau_{x}>t)} \int_a^b f(r,x)dr \\ \nonumber
   &+ \int_a^b \int_{0}^{t}
   \frac{\mathbb{E}^{0}({\bf 1}_{\tau_{x}>u}L_{u}h(X_{u})f(r-u,x-X_u)|\mathcal{F}_{u}^{Q})}
   {\mathbb{E}^{0}({\bf 1}_{\tau_{x}>t}L_{u}|\mathcal{F}_{u}^{Q})}dQ_{u}dr\\ \nonumber
   &- \int_a^b \int_{0}^{t}
   \frac{\mathbb{E}^{0}(L_{u}{\bf 1}_{\tau_x>u}f(r-u,x-Xu)|\mathcal{F}_{u}^{Q})
   \mathbb{E}^{0}({\bf 1}_{\tau_{x}>u}L_{u}h(X_{u})G(t-u,x-X_{u})|
   \mathcal{F}_{u}^{Q})} {[\mathbb{E}^{0}({\bf 1}_{\tau_{x}>t}L_{u}|\mathcal{F}_{u}^{Q})]^{2}}dQ_{u}dr\\ \nonumber
   &+\int_a^b \int_{0}^{t}
   \frac{\mathbb{E}^{0}(L_{u}{\bf 1}_{\tau_x>u}f(r-u,x-X_u)|\mathcal{F}_{u}^{Q})
   [\mathbb{E}^{0}({\bf 1}_{\tau_{x}>u}L_{u}h(X_{u})G(t-u,x-X_{u})|
   \mathcal{F}_{u}^{Q})]^{2}} {[\mathbb{E}^{0}({\bf 1}_{\tau_{x}>t}L_{u}|\mathcal{F}_{u}^{Q})]^{3}}dudr\\ \nonumber
   &-\int_a^b \!\! \int_{0}^{t}
  \mathbb{E}^{0}({\bf 1}_{\tau_{x}>u}L_{u}h(X_{u})f(r\!\!-\!\!u,x\!\!-\!\!X_u)|\mathcal{F}_{u}^{Q})
  \frac{\mathbb{E}^{0}({\bf 1}_{\tau_{x}>u}L_{u}h(X_{u})G(t\!-\!u,x\!\!-\!\!X_{u})|\mathcal{F}_{u}^{Q})}
   {[\mathbb{E}^{0}({\bf 1}_{\tau_{x}>t}L_{u}|\mathcal{F}_{u}^{Q})]^{2}}dudr . \nonumber
  \end{align*}
 To express this result with $\mathbb{P}$ conditional expectation instead of $\mathbb{P}^{0}$ conditional
 expectation,each fraction under the integral is multiplied and divided by the same term
 $\mathbb{E}^{0}({\bf 1}_{\tau_{x}>u}L_{u}|\mathcal{F}_u^Q).$To manage the indicator function, we use the filtration
 $\mathcal{G}_{t}$ because $\tau_x$ is a $\mathcal{G}_{t}-$ stopping time.\\ 
 Therefore, on the set $\{\tau_{x}>t\},$ 
 \begin{align*}
  \mathbb{E}({\bf 1}_{a<\tau_{x}<b}|  \mathcal{G}_{t})= &\frac {1}{\mathbb{P}(\tau_{x}>t)}\int_{a}^{b} f(r,x)dr
 +\int_{a}^{b} \int_{0}^{t} \frac{\mathbb{E}({\bf 1}_{\tau_{x}>u}h(X_{u})f(r-u,x-X_{u})|\mathcal{G}_{u})}
   {\mathbb{E}({\bf 1}_{\tau_{x}>u}G(t-u,x-X_{u})|\mathcal{G}_{u})}dQ_{u}dr\\ 
   \displaystyle   
   &-\int_{a}^{b}\int_{0}^{t} \bar{f}(r,u,x)
   \frac{\mathbb{E}({\bf 1}_{\tau_{x}>u}h(X_{u})G(t-u,x-X_{u})|
   \mathcal{G}_{u})} {[\mathbb{E}({\bf 1}_{\tau_{x}>u}G(t-u,x-X_{u})|\mathcal{G}_{u})]^{2}}dQ_{u}dr\\ 
   \displaystyle
   &+\int_{a}^{b}\int_{0}^{t} \bar{f}(r,u,x)\frac{
   [\mathbb{E}({\bf 1}_{\tau_{x}>u}h(X_{u})G(t-u,x-X_{u})|
   \mathcal{G}_{u})]^{2}} {[\mathbb{E}({\bf 1}_{\tau_{x}>u}G(t-u,x-X_{u})|\mathcal{G}_{u})]^{3}}dudr\\ 
   \displaystyle
   &-\int_{a}^{b}\int_{0}^{t} \frac{\mathbb{E}({\bf 1}_{\tau_{x}>u}h(X_{u})f(r-u,x-X_{u})|\mathcal{G}_{u}) 
 \mathbb{E}({\bf 1}_{\tau_{x}>u}h(X_{u})G(t-u,x-X_{u})|\mathcal{G}_{u})}
   {[\mathbb{E}({\bf 1}_{\tau_{x}>u}G(t-u,x-X_{u})|\mathcal{G}_{u})]^{2}}du dr .
 \end{align*}
 \end{proof}

 \begin{rems}
  \begin{equation*}
   \mathbb{E}\left(\int_{0}^{t} \frac{du}{\mathbb{E}({\bf 1}_{\tau_{x}>u}G(t-u,x-X_{u})|\mathcal{G}_{u})^{2}}\right)<\infty.
  \end{equation*}
Indeed, let $Z_{u}=\mathbb{E}({\bf 1}_{\tau_{x}>u}G(t-u,x-X_{u})|\mathcal{G}_{u}).$ We have 
$\mathbb{E}({\bf 1}_{\tau_{x}>t}|\mathcal{G}_{u})= \mathbb{E}\left(\mathbb{E}({\bf 1}_{\tau_{x}>t}|\mathcal{F}_{u})|
\mathcal{G}_{u})\right).$ applying $(\mathcal{F},\mathbb{P})$ Markov property at point $u$ to ${\bf 1}_{\tau_x >t},$ 
 it follows that $\mathbb{E}({\bf 1}_{\tau_{x}>t}|\mathcal{G}_{u})=
\mathbb{E}({\bf 1}_{\tau_{x}>u}G(t-u,x-X_{u})|\mathcal{G}_{u}).$ Since $\mathbb{E}({\bf 1}_{\tau_{x}>t}|\mathcal{G}_{u})$
is a martingale then $(Z_{u}^{2})_{u\geq 0} $ is submartingale with increasing expectation. This yields
\begin{equation*}
 \mathbb{E}\left(\int_{0}^{t} \frac{du}{\mathbb{E}({\bf 1}_{\tau_{x}>u}G(t-u,x-X_{u})|\mathcal{G}_{u})^{2}}\right)\leq 
 t \mathbb{P}(\tau_{x}>t).
\end{equation*}
 \end{rems}
 \end{proof}
 
\section{Appendix}
\begin{lemme}
\label{lemme5}
 Let $G$ be a Gaussian random variable $\mathcal{N}(0,1)$ and let $m, \mu \in \mathbb{R}, t, \sigma \in \mathbb{R}_{+}.$ Then
 \begin{equation*}
A(\mu, \sigma,m,t):= \mathbb{E}\left(\frac{[\mu-\sigma G+mt]_{+}}{\sqrt{2 \pi t^{3}}}e^{-\frac{(\mu-\sigma G)^{2}}{2t}}\right)
 \end{equation*}
satisfies
\begin{equation}
  A(\mu, \sigma,m,t)\leq \frac{C_{1}}{(\sigma^{2}+t)^{\frac{3}{2}}}+\frac{C_{2}}{\sqrt{\sigma^{2}+t}}+
  \frac{\sigma C_{3}}{(\sigma^{2}+t)\sqrt{t}}
\end{equation}
whith $C_{1},C_{2},$ and $ C_{3}$ positive constantes.
\end{lemme}
\begin{proof}
 We use the law of $G$ and it follows that
 \begin{equation*}
 A(\mu, \sigma,m,t)= \frac{1}{\sqrt{2\pi}}\int_{\mathbb{R}}
 \frac{[\mu-\sigma y+mt]_{+}}{\sqrt{2 \pi t^{3}}}e^{-\frac{(\mu-\sigma y)^{2}}{2t}}e^{-\frac{y^{2}}{2}}dy.
 \end{equation*}
 Since $ \frac{(\mu-\sigma y)^{2}}{t} + y^{2}= \frac{\sigma^{2}+t}{t}\left(y-\frac{\mu \sigma}{\sigma^{2}+t}\right)^{2}
 + \frac{\mu^{2}}{\sigma^{2}+t}$, then
 \begin{equation*}
 A(\mu, \sigma,m,t)= \frac{1}{\sqrt{2\pi}}\int_{\mathbb{R}}
 \frac{[\mu-\sigma y+mt]_{+}}{\sqrt{2 \pi t^{3}}}\exp\left[-\frac{\sigma^{2}+t}{2t}\left(y-\frac{\mu \sigma}{\sigma^{2}+t}\right)^{2}
 - \frac{\mu^{2}}{2(\sigma^{2}+t)}\right]dy.
 \end{equation*}
By a change of variable $z= y-\frac{\mu \sigma}{\sigma^{2}+t},$ we have
\begin{equation*}
 A(\mu, \sigma,m,t)= \frac{1}{\sqrt{2\pi}}\int_{\mathbb{R}} 
 \frac{[\mu-\sigma (z+\frac{\mu \sigma}{\sigma^{2}+t})+mt]_{+}}{\sqrt{2 \pi t^{3}}}
 \exp\left[-\frac{\sigma^{2}+t}{2t}z^{2}- \frac{\mu^{2}}{2(\sigma^{2}+t)}\right]dz.
\end{equation*}
A new change of variable $y=z\sqrt{\frac{\sigma^{2}+t}{t}} $ leads us to 
\begin{equation*}
 A(\mu, \sigma,m,t)= \frac{1}{\sqrt{2\pi}}\int_{\mathbb{R}} \frac{1}{(\sigma^{2}+t)\sqrt{2 \pi t}}
 \left(\sqrt{t}\left[\frac{\mu}{\sqrt{\sigma^{2}+t}}+m\sqrt{\sigma^{2}+t}\right]-\sigma y\right)_{+}
 \exp\left[-\frac{y^{2}}{2}- \frac{\mu^{2}}{2(\sigma^{2}+t)}\right]dy
\end{equation*}
After some calculation, we obtain
\begin{align*}
  A(\mu, \sigma,m,t)&=\frac{\exp[-\frac{\mu^{2}}{2(\sigma^{2}+t)}]}{(\sigma^{2}+t)\sqrt{2\pi}}
  \left(\frac{\mu}{\sqrt{\sigma^{2}+t}}+ m\sqrt{\sigma^{2}+t}\right)\Phi\left( \frac{\sqrt{t}}{\sigma}\left[\frac{\mu}
  {\sqrt{\sigma^{2}+t}}+m\sqrt{\sigma^{2}+t}\right]\right)\\
  & + \frac{\sigma\exp[-\frac{\mu^{2}}{2(\sigma^{2}+t)}]}{4\pi(\sigma^{2}+t)\sqrt{t}}
  \exp\left[-\frac{t}{2\sigma^{2}}\left(\frac{\mu}{\sqrt{\sigma^{2}+t}}+m\sqrt{\sigma^{2}+t}\right)\right]
\end{align*}
where $\Phi$ is the Gaussian distribution function which is bounded by $1$. Now, we get 
\begin{equation}
  A(\mu, \sigma,m,t)\leq \frac{C_{1}}{(\sigma^{2}+t)^{\frac{3}{2}}}+\frac{C_{2}}{\sqrt{\sigma^{2}+t}}+
  \frac{\sigma C_{3}}{(\sigma^{2}+t)\sqrt{t}}
\end{equation}
whith $C_{1},C_{2},$ and $ C_{3}$ positive constantes. 
\end{proof}
\begin{lemme}
\label{lemme6}
 If $(T_i , i \in \mathbb{N}^{*})$ is the sequence of jump time of the process $N $, then
 \begin{equation*}
  \mathbb{E}\left(\sqrt{\frac{T_{N_{t}}}{t-T_{N_{t}}}}\right) < 2 \lambda t.
 \end{equation*}
\end{lemme}
\begin{proof}
 We have
 \begin{align*}
 \displaystyle \mathbb{E}\left(\sqrt{\frac{T_{N_{t}}}{t-T_{N_{t}}}}\right) &= \sum_{n\geq 1} \mathbb{E}
 \left( \sqrt{\frac{T_n}{t-T_n}}{\bf 1}_{T_{n}<t<T_{n+1}}\right)\\
  &=\displaystyle \sum_{n\geq 1}\int_{0}^{t}\sqrt{\frac{u}{t-u}}\frac{(\lambda u)^{n-1}}{(n-1)!}\lambda e^{-\lambda u}
  \int_{t-u}^{+\infty}\lambda e^{-\lambda v} dv du\\
  &=\displaystyle e^{-\lambda t}\sum_{n \geq 1}\frac{\lambda^{n}}{(n-1)!}\int_{0}^{t}\frac{u^{n-\frac{1}{2}}}
  {(t-u)^{\frac{1}{2}}} du \\
  & \leq \displaystyle e^{-\lambda t}\sum_{n \geq 1}\frac{\lambda^{n}t^{n-\frac{1}{2}}}{(n-1)!}\int_{0}^{t}\frac{du}
  {(t-u)^{\frac{1}{2}}} du \\
  & \leq 2 \lambda t e^{-\lambda t}\sum_{n \geq 0}\frac{(\lambda t)^{n}}{n!} = 2 \lambda t.
  \end{align*}

\end{proof}
\begin{lemme}
\label{lemme7}
 Let  $Y\in L^\infty(\Omega,\mathbb{P},{\cal F}^X_T)$ and  $T\geq t$, then 
  $$ \mathbb{E}^{0}(YL_T S_t)= 
  \mathbb{E}^{0}(Y)+ 
  \mathbb{E}^{0}\left(
  \int_{0}^{t}\mathbb{E}^{0}(Y/\mathcal{F}_{u})S_u\rho_u L_u h(X_u)du \right) $$
  and
  $$\mathbb{E}^{0}(Y L_T|\mathcal{F^Q}_{t})= \mathbb{E}^{0}(Y L_t|\mathcal{F}^Q_{t})~;~ \mathbb{E}^{0}(Y L_T|\mathcal{F}_{t})= \mathbb{E}^{0}(Y L_t|\mathcal{F}_{t}).  $$
  For instance
  $$ \mathbb{E}^{0}({\bf 1}_{\tau_x >T}L_T|\mathcal{F}_{t})= \mathbb{E}^{0}({\bf 1}_{\tau_x >T}L_t|\mathcal{F}_{t}).  $$

\end{lemme}
\begin{proof}
 Let $ t \leq T$
 and $S_{t}= \int_{0}^{t}S_u\rho_u dQ_u\in \mathcal{S}_t $ and define the process $K$
 \begin{equation*}
  K_. =1 + \int_{0}^{.} {\bf 1}_{u\leq t}S_u\rho_u dQ_u.
 \end{equation*}

 The integration by parts formula of It\^o applying to  the product $L_.K_.$ between $0$ and $T$
 permits us to obtain
 \begin{align*}
  &L_T K_T =1 + \int_{0}^{T}{\bf 1}_{u\leq t}L_u S_u \rho_u dQ_u + \int_{0}^{T}K_u L_u h(X_u)dQ_u + \int_{0}^{T}
  {\bf 1}_{u\leq t}S_u L_u \rho_u h(X_u)du
 \end{align*}
 and remark that $L_TK_T=L_TS_t.$

 Since $X$ and $Q$ are independent under $\mathbb{P}^{0},$ it follows 
 \begin{align}
 \label{IPPI}
 \displaystyle \mathbb{E}^{0}(YL_T S_t)&= \displaystyle\mathbb{E}^{0}(Y)+ \mathbb{E}^{0}\left(
  \int_{0}^{t\wedge T} \mathbb{E}^{0}(Y/\mathcal{F}_{u})S_u\rho_u L_u h(X_u) du \right)
  \nonumber \\
  &=\displaystyle \mathbb{E}^{0}(Y)+ \mathbb{E}^{0}\displaystyle \left(
  \int_{0}^{t} YS_u\rho_u L_u h(X_u) du \right). 
   \end{align}
  Similarly, using first $\mathbb{E}^{0}[YL_t S_t]=
   \mathbb{E}^{0}[\mathbb{E}^{0}(Y/\mathcal{F}_{t})L_t S_t],$
  It\^o's formula on product of processes $\mathbb{E}^{0}(Y/\mathcal{F}_{.})L_. S_.$ and the independence 
  beetwen $X$ and $Q$ under $\mathbb{P}^0$ yields
   \begin{equation}\label{IPP}
   \displaystyle \mathbb{E}^{0}(YL_t S_t) =
    \displaystyle \mathbb{E}^{0}(Y)+ \mathbb{E}^{0}\displaystyle \left(
  \int_{0}^{t} YS_u\rho_u L_u h(X_u) du \right)
   \end{equation}
   (\ref{IPPI}) and (\ref{IPP}) imply that 
   \begin{equation*}
    \mathbb{E}^{0}(YL_T|\mathcal{F}^Q_{t})= \mathbb{E}^{0}(YL_t|\mathcal{F}^Q_{t}).
   \end{equation*}
 \\
  Now let $f_t(X)\in  L^\infty(\Omega,\mathbb{P}^{0},\mathcal{F}^X_{t})$
 and apply the above equality to $Yf_t(X)$:
$$  \mathbb{E}^{0}(Yf_t(X)L_T|\mathcal{F}^Q_{t})= \mathbb{E}^{0}(Yf_t(X)L_t|\mathcal{F}^Q_{t})$$
so
$$ \mathbb{E}^{0}(Yf_t(X)L_TS_t)=\mathbb{E}^{0}(Yf_t(X)L_tS_t)$$
which concludes the proof.
\end{proof}

\end{document}